\documentclass[12pt]{article}
\usepackage{amssymb,amsmath,amsthm,latexsym, bbm}
\usepackage{enumerate}
\usepackage{color}
\usepackage{graphicx}
\usepackage{mathrsfs}
\usepackage[mathscr]{euscript}
\usepackage{amscd}

\newtheorem{thm}{Theorem}[section]

\newtheorem{lemma}[thm]{Lemma}

\newtheorem{remark}[thm]{Remark}
\newtheorem{example}[thm]{Examples}
\numberwithin{equation}{section}

\def\Ex {{\mathbb E}}

\def\R {{\mathbb R}}
\def\Z {{\mathbb Z}}

\def\rd {{\mathrm d}}
\def\F {{\mathcal F}}
\def\P {{\mathbb P}}
\def\Q {{\mathbb Q}}

\def\wt{\widetilde}
\def\wh{\widehat}

\def\op{\mathcal L}
\def\1{{\mathbf 1}}

\def\<{{\langle}}
\def\>{{\rangle}}

\def\wh{\widehat}

\def\proof{{\medskip\noindent {\bf Proof. }}}

\def\qed{{\hfill $\square$ \bigskip}}

\def\Levy {{L\'evy\ }}
\def\Schrodinger {{Schr\"odinger\ }}

\begin{document}
\title{Inverse local time of one-dimensional diffusions and its comparison theorem}

\author{{Zhen-Qing Chen} \quad \hbox{and} \quad {Lidan Wang}}

\date{\today }

\maketitle

\begin{abstract}
 In this paper, we study inverse local time at 0 of one-dimensional reflected diffusions on $[0, \infty)$,
and establish a comparison principle for inverse local times. Applications to Green function estimates for non-local operators are given. 
\end{abstract}



\section{Introduction}
\label{sec:intr}

It is well known that the trace of Brownian motion in $\R^{d+1}$ (or reflected Brownian motion in the upper half space $\R^{d+1}_+:=\{x=(x_1, \dots, x_d, x_{d+1})\in \R^{d+1}: x_{d+1}>0\}$)
 on the hyperplane $\{x_{d+1}=0\}$ is a $d$-dimensional Cauchy process.
Molchanov and Ostrovski \cite{MO} in 1969 showed that in fact any rotationally symmetric 
$\alpha$-stable process
on $\R^d$ with $0<\alpha <2$ can be realized as the boundary trace on $\{x_{d+1}=0\}$ 
of some reflected diffusion in the upper half space $\R_+^{d+1}$. In terms of the generators, Molchanov and Ostrovski's result says that 
the fractional Laplacian $\Delta^{\alpha/2}:=-(-\Delta)^{\alpha/2}$ in $\R^{d+1}$ can be realized
as the boundary trace of some (degenerate) differential operator in upper half space of one-dimensional higher. The latter fact is rediscovered by Caffarelli and Silvestre \cite{CS} in 2007 using a purely analytic approach.
Realizing non-local operators as boundary trace of some differential operators is a powerful way to study non-local operators 
from analytic point of view as one can employ many well developed techniques and ideas from partial differential equations 
(PDE). 
It is a natural and interesting question to investigate the scope of  non-local operators that can be realized
as the boundary trace of differential operators.

Note that rotationally symmetric stable process has the same distribution as Brownian motion time-changed by an independent stable subordinator. So the key to Molchanov and Ostrovski's result is to show that $(\alpha/2)$-stable subordinator can be realzed as the inverse local time at 0 of some one-dimensional reflected diffusion on $[0, \infty)$.  Indeed,  
Molchanov-Ostrovski \cite{MO} showed that for $0<\alpha<1$, $\alpha$-stable subordinator can be realized as the inverse local time of a reflecting Bessel process on $[0, \infty)$ determined locally by the generator
\begin{equation} \label{e:1.1}
\op^{(\alpha)}=\frac{1}{2}\frac{d^2}{dx^2}+\frac{1-2\alpha}{2x}\frac{d}{dx}.
\end{equation}

The class of subordinate Brownian motions that can be realized as boundary traces of diffusion processes in upper half space
of one-dimensional higher is in one-to-one correspondence with the class of subordinators 
that can be realized as the inverse local time
at 0 of some reflected diffusions on the half line $[0, \infty)$. The latter is exactly a question raised in It\^{o}-McKean \cite{IM} called Krein representation problem. This problem remains open.
However,  Knight \cite{Knight} and Kotani-Watanabe \cite{KW} showed independently in 1981-1982 that 
if one relaxes $\R_+$-valued diffusions to gap diffusions (a family of Markov processes having possibly discontinuous trajactories), then the answer is affirmative. It is a general fact that the inverse local time of a Markov process
as a point having positive capacity is always a subordinator, that is, a non-decreasing real-valued L\'evy process.

Relativistic Cauchy process in $\R^d$ (also called relativistic Brownian motion in the study of relativistic Hamiltonian system in physics \cite{CMS}) is a subordinate Brownian motion $X_t$ characterized by 
$$
\Ex e^{ i\xi \cdot (X_t-X_0)} = e^{ t \left( \sqrt{m}-\sqrt{ m+|\xi|^2 } \right) }, \quad \xi \in \R^d , 
$$  
where $m>0$ stands for the mass of the particle. The  infinitesimal generator of $X_t$ is $\sqrt{m}-\sqrt{m-\Delta}$. 
It is not hard to see that the inverse local time at $0$ of 
the reflected Brownian motion with downward constant drift $\sqrt{2m}$ on $[0, \infty)$
is a subordinator $S_t$ with $S_0=0$ and 
$$ \Ex e ^{-\lambda S_t} = e^{ t \left( \sqrt{m}-\sqrt{m+\lambda } \right)}, \quad \lambda \geq 0.
$$
 Hence the relativisitc Cauchy process on $\R^d$ can be regarded as the boundary trace
on $\{x_{d+1}=0\}$ in the upper half space of $\R^{d+1}$ 
where the vertical motion in the $x_{d+1}$ direction is a Brownian motion with downward constant drift while the horizontal motion is an independent Brownian motion in $\R^d$. 
However for general $\alpha \not= 1/2$, relativistic $\alpha$-stable processes can not be simply realized,
by analogy with rotationally symmetric stable processes,  
as a diffusion in the upper half space of one-dimensional high whose vertical motion is a Bessel process with constant drift.

 By using Esscher transform, Martin and Yor\cite{Yor1} showed that relativistic $\alpha$-subordinator for $0<\alpha<1$
is in fact the inverse local time at $0$ of the reflected diffusion on $[0, \infty)$ determined by generator 
\begin{equation}\label{e:1.2}
\op^{(\alpha,m)}=\frac{1}{2}\frac{d^2}{dx^2}+\bigg(\frac{1-2\alpha}{2x}+\frac{\wh K'_\alpha(\sqrt{2m}x)}{\wh K_\alpha(\sqrt{2m}x)}\bigg)\frac{d}{dx},
\end{equation} 
where $\wh K(x)=x^\alpha K_{\alpha}(x)$ and $K_\alpha(x)$ is the modified Bessel function of the second kind defined in \eqref{eqn:Kalpha}.

In this paper, we set out a modest goal to investigate properties of the inverse local time at $0$ of 
reflected diffusions on $[0, \infty)$ with infinitesimal generator of the form
\begin{equation}\label{e:1.3}
\op=\frac{1}{2}\frac{d^2}{d x^2}+\left(\frac{1-2\alpha}{2x}-f(x)\right)\frac{d}{d x},
\end{equation}
where $f\geq 0$ is a function on $(0, \infty)$, and the corresponding subordinate Brownian motions.

For two functions $f$ and $g$, notation $f\lesssim g$ means there is a constant $c>0$
so that $f\leq c g$. The following are the main results of this paper.

\begin{thm}\label{T:1.1}
Let $Y_t$ be the reflected diffusion process on $[0,\infty)$ determined by the local generator
of the form \eqref{e:1.3} with 
$$
0\leq f(x)\lesssim (1\wedge x)^{2\alpha -1} \quad \hbox{on } (0, \infty) .
$$
Let $S_t$ be the inverse local time of $Y_t$ at $0$. Then there is a constant $m>0$ so that
stochastically 
$S_t^{(\alpha,m)}\leq S_t\leq S^{(\alpha)}_t$ for all $t\geq 0$,
where $S^{(\alpha)}_t$ and $S_t^{(\alpha,m)}$ are the $\alpha$-stable subordinator and
relativistic $\alpha$-stable subordinator with mass $m$, respectively. 
\end{thm}

As an application,  we have the following Green function estimates for the trace processes of diffusion processes in $\R^{d+1}$ whose vertical $x_{d+1}$-coordinate is a reflected diffusion on $[0, \infty)$ with infinitesimal generator \eqref{e:1.3}, and the horizontal direction is an independent Brownian motion
in $\R^d$.

\begin{thm} \label{T:1.2}
Under the setting of Theorem \ref{T:1.1}, let $B_t$ be a $d$-dimensional Brownian motion independent of $Y_t$
with variance $2t$, and $\mu(x)$ be the density of the \Levy measure of trace process $B_{S_t}$. 
Denote the density of the \Levy measure of symmetric $2\alpha$-stable process  by $\mu^{(\alpha)}(x)$. Then $j(x) :=\mu^{(\alpha)}(x)-\mu(x)\geq 0$, and there exists a constant $C$ such that for $|x|\leq 1$,
$$j(x)\leq C|x|^{2-2\alpha -d}.$$
Let $D\subset\R^d$ be a bounded connected Lipschitz open set. Denote Green functions of the  trace process $B_{S_t}$ in $D$ by $G_D(x,y)$. Then there exists a constant $C_1=C_1( d, \alpha, D,C)$ such that
$$
C_1^{-1} G^{(2\alpha)}_D(x,y)\leq G_D(x,y)\leq C_1G^{(2\alpha)}_D(x,y)
\text{ for }  x,y\in D,
$$
where $G^{(2\alpha)}_D$ is the Green function of rotationally symmetric $(2\alpha)$-stable process,
or equivalently of the fractional Laplacian $\Delta^\alpha$,  in $D$.
\end{thm}

The rest of the paper is organized as follows. 
 Esscher transform (an Economics terminology) and Girsanov transform for reflected diffusions
on $[0, \infty)$ are discussed in Section \ref{sec:EG}.
This extends and refines the corresponding part of Martin-Yor \cite{Yor1} with a more complete and rigorous proof. 
In Section \ref{sec:com}, we present a key comparison result for the inverse local times at $0$
for reflected diffusion processes on $[0, \infty)$ and their corresponding L\'evy measures. 
  Hausdorff measure of zero sets and Girsanov transform are the main tools to get this result. Regenerative embedding theory for subordinators 
are also used.
With the comparison result obtained in Section \ref{sec:com},  Theorems \ref{T:1.1}
and \ref{T:1.2} are established in Section \ref{S:4}.

\section{Esscher  transforms}
\label{sec:EG}

Recall that the Laplace exponent and L\'evy measure for $\alpha$-stable subordinator,
where $0<\alpha <1$,  
are $\phi^{(\alpha)}(\lambda)=c_\alpha\lambda^\alpha$ and $\nu^{(\alpha)}(dx):=c_\alpha x^{-1-\alpha} dx$, where $c_\alpha = \alpha/\Gamma (1-\alpha)$;
while that for relativisitic $\alpha$-stable subordinator with mass $m>0$ are $\phi^{(\alpha,m)}(\lambda)=c_\alpha[(m +\lambda)^{\alpha}-m^\alpha]$ and  $\nu^{(\alpha, m)} (dx):= c_{\alpha }x^{-1-\alpha} e^{-mx} dx$. 
 
 Fix $0<\alpha <1$. 
 Let $(\Omega,\F, \mathbb P)$ be a probability space on which a reflected Bessel process $Y_t$ on 
$[0, \infty)$ with generator \eqref{e:1.1} is defined.  The filtration generated by $Y_t$ will be denoted 
as $\{\F_t; t\geq 0\}$.
Let $L_t$ be the local time of $Y$ at $0$, and 
$$ 
S_t:= \inf\{s>0: L_s >t\} , 
$$
the inverse of $L$, which is a   stopping time with respect to the filtration
$\{\F_t; t\geq 0\}$. 
We know that $S_t$ is an $\alpha$-stable subordinator. 
We can define a new probability measure $\Q$ on $(\Omega,\F)$ by 
\begin{equation}\label{e:2.1} 
\frac{\Q(dx)}{\P (dx)} = \frac{e^{-mx}}{\Ex_0[\exp(-mS_t  ]}  
= \exp(tc_\alpha m^\alpha-mx) \quad \hbox{ on } \F_{S_t}.
\end{equation}
This change of measure is called Esscher transform in literature (see Chapter VII, 3c, \cite{E}). 
Note that  under the new probability measure $\Q$,
\begin{align}
\Ex^\Q [\exp(-\lambda S_t ]=&\exp(tc_\alpha m^\alpha)\Ex^\P[\exp(-(\lambda+m)S_t ]\notag\\
=&\exp(tc_\alpha m^\alpha-t\phi^{(\alpha)(\lambda+m)})\notag\\
=&\exp(-t\phi^{(\alpha,m)}(\lambda)).
\end{align}
In other words, under $\Q $, $\{S_t ; t\geq 0\}$ is  a relativistic $\alpha$-stable subordinator with mass $m$.

\medskip

We now  extend the above Esscher transform to   general one-dimensional diffusions. 

\begin{thm}
\label{thm:EG}
Suppose that $X_t$ is a reflected diffusion process on $[0, \infty)$ defined on a probability space   $(\Omega,\F, \mathbb P)$, 
determined locally by the generator 
$$\op=a(x)\frac{d^2}{d x^2}+b(x)\frac{d}{d x}.$$ 
Let $L_t$ be the local time of $X_t$ at $0$, and $S_t=\inf\{s:L_s>t\}$ its  inverse local time at $0$, which is a subordinator.
Denote by $\phi(\lambda)$ the  Laplace exponent of $\{S_t; t\geq 0\}$.  Define
\begin{equation} \label{e:2.3}
\frac{d \Q }{d \P }:=\frac{\exp(-mS_t)}{\Ex [\exp(-mS_t)]} \quad \hbox{on  }\F_{S_t}, \ t\geq 0. 
\end{equation}
Then
\begin{itemize}
\item[\rm(i)]  \eqref{e:2.3}  defines a new measure $\Q $ on $\F_\infty$ in a consistent way;

\item[\rm(ii)] Under $\Q$, the original diffusion $X$, write as $X^{(m)}$ for emphasis, is a reflected diffusion on $[0, \infty)$
having  generator
$$
\op^{(m)}=\op+2a(x)\frac{\rho'_m(x)}{\rho_m(x)}\frac{d}{d x}, \ x>0, 
$$
where $\rho_m(x):=\Ex_x[\exp(-mT_0)]$ and $T_0$ is the first hitting time of $0$ by the process $X_t$.
 
\item[\rm(iii)] Denote by $S_t^{(m)}$  the inverse local time of $X_t^{(m)}$ at $0$ and $\phi^{(m)}(\lambda)$ the Laplace exponent
for subordinator $S_t^{(m)}$. Then we have
$$
\phi^{(m)}(\lambda)=\phi(\lambda+m)-\phi(m).
$$
\end{itemize}
\end{thm}

\proof The definition (\ref{e:2.3}) gives 
\begin{equation}
\label{eqn:gir2}
M_{s,t}:=\frac{d\Q }{d\P }\bigg|_{\F_{s\wedge S_t}}=\frac{\Ex [\exp(-mS_t)|\F_{s\wedge S_t}]}{\Ex [\exp(-mS_t)]}.
\end{equation}
To see the consistency, we observe for $r\leq t$, since $S_t$ is a subordinator,
\begin{align*}
\frac{d\Q }{d\P }\bigg|_{\F_{s\wedge S_r}}=&\frac{d\Q }{d\P }\bigg|_{\F_{s\wedge S_r\wedge S_t}}=\frac{\Ex [\exp(-mS_t)|\F_{s\wedge S_r\wedge S_t}]}{\Ex [\exp(-mS_t)]}\\
=&\frac{\Ex [\exp(-mS_t)|\F_{s\wedge S_r}]}{\Ex [\exp(-mS_t)]}\\
=&\frac{\Ex [\exp(-mS_r)\exp(-m(S_t-S_r))|\F_{s\wedge S_r}]}{\Ex [\exp(-mS_r)\exp(-m(S_t-S_r)]}\\
=&\frac{\Ex [\exp(-m(S_t-S_r))]\Ex [\exp(-mS_r)|\F_{s\wedge S_r}]}{\Ex [\exp(-m(S_t-S_r))]\Ex [\exp(-mS_r)]}\\
=&\frac{\Ex [\exp(-mS_r)|\F_{s\wedge S_r}]}{\Ex [\exp(-mS_r)]}
\end{align*}
To see uniform integrability, we first have the decomposition of the inverse local time at $0$, since $S_0=T_0$, 
\begin{align}
\label{eqn:inv}
S_t\text{ under }\P =&\inf\{s:L_s>t \text{ with }X_0=x\}\notag\\
=&\inf\{s:s=T_0+r,L_{T_0}+L_r\circ\theta_{T_0}>t\text{ with }X_0=x\}\notag\\
=&T_0+\inf\{r:L_r\circ\theta_{T_0}>t\text{ with }X_0=x\}\notag\\
=&T_0+\inf\{r:L_r>t\text{ with }X_0=0\}\notag\\
=&T_0+S_t\text{ under }\P_0
\end{align}
With the decomposition,
\begin{align}
\label{eqn:exp}
\Ex [\exp(-mS_t)]=&\Ex [e^{-mT_0}]\Ex_0[e^{-mS_t}]\notag\\
=&\rho_m(x)\Ex_0[e^{-mS_t}]\notag\\
=&\rho_m(x)\exp[-t\phi(m)]
\end{align}
Similarly,  
\begin{align*}
\1_{\{s\leq S_t\}}\Ex [e^{-mS_t}|\F_{s}]=&\1_{\{s\leq S_t\}}\Ex [e^{-mS_t}|\F_s]\\
=&\1_{\{s\leq S_t\}}e^{-ms}\Ex_{X_s}[\exp(-mS_{t-r})]|_{r=L_s}
\end{align*}
The last equation holds because on $\{s\leq S_t\}$,
\begin{align*}
S_t=&\inf\{r+s:L_{r+s}>t\}=s+\inf\{r:L_s+L_r\circ\theta_s>t\}\\
=&s+\inf\{r:L_r\circ\theta_s>t-L_s\}\\
=&s+S_{t-r}\circ\theta_s|_{r=L_s}.
\end{align*}
Now using (\ref{eqn:exp}), we have
\begin{align*}
\Ex_{X_s}[\exp(-mS_{t-r})]|_{r=L_s}=&\rho_m(X_s)\Ex_0[\exp(-mS_{t-r})]|_{r=L_s}\\
=&\rho_m(X_s)\exp(-(t-r)\phi(m))|_{r=L_s}\\
=&\rho_m(X_s)\exp(-(t-L_s)\phi(m)).
\end{align*}
Thus, restricted to $\F_{s\wedge S_t}$,
\begin{align*}
M_{s,t}=&\frac{d\Q }{d\P }\bigg|_{\F_{s\wedge S_t}}=\1_{\{s\leq S_t\}} \frac{\Ex [\exp(-mS_t)|\F_{s}]}{\Ex [\exp(-mS_t)]}+\1_{\{s>S_t\}}\frac{\exp(-mS_t)}{\Ex [\exp(-mS_t)]}\\
=&\1_{\{s\leq S_t\}}\frac{e^{-ms}\rho_m(X_s)\exp(-(t-L_s)\phi(m))}{\rho_m(x)\exp(-t\phi(m))}+\1_{\{s>S_t\}}\frac{\exp(-mS_t)}{\Ex [\exp(-mS_t)]}\\
=&\1_{\{s\leq S_t\}}\frac{\rho_m(X_s)}{\rho_m(x)}\exp(-ms+L_s\phi(m))+\1_{\{s>S_t\}}\frac{\exp(-mS_t)}{\Ex [\exp(-mS_t)]}
\end{align*}
As $t\to\infty$, $M_{s,t}\to M_s$, a.s. and 
$$M_s:=\frac{\rho_m(X_s)}{\rho_m(x)}\exp(-ms+L_s\phi(m)).$$
It's obvious that $M_s\in\F_s$. Also, from the original definition of $M_{s,t}$ in (\ref{eqn:gir2}), for any $s,t$, $\Ex  M_{s,t}\leq 1$, so by Fatou's lemma,
$$\Ex M_s\leq\liminf\Ex M_{s,t}\leq 1.$$
Thus, $M_s\in L^1$, and on the other hand, with $\F_{s\wedge S_t}\subset\F_s$
\begin{align*}
\Ex [M_s|\F_{s\wedge S_t}]=&\Ex \bigg[\frac{\rho_m(X_s)}{\rho_m(x)}\exp(-ms+L_s\phi(m))\Big|\F_{s\wedge S_t}\bigg]\\
=&\1_{\{s\leq S_t\}}\frac{\rho_m(X_s)}{\rho_m(x)}\exp(-ms+L_s\phi(m))+\1_{\{s>S_t\}}\frac{\rho_m(0)}{\rho_m(x)}\exp(-mS_t+t\phi(m)])\\
=&\1_{\{s\leq S_t\}}\frac{\rho_m(X_s)}{\rho_m(x)}\exp(-ms+L_s\phi(m))+\1_{\{s>S_t\}}\frac{\exp(-mS_t)}{\Ex (\exp(-mS_t))}\\
=&M_{s,t},
\end{align*} 
the second to the last equality comes from $\rho_m(0)=1$ and (\ref{eqn:exp}). Thus $\{M_{s,t}=\Ex [M_s|\F_{s\wedge S_t}]\}_{t\geq0}$ is uniformly integrable. Taking $t\to \infty$ yields 
\begin{equation} \label{eqn:gir3}
\frac{d\Q }{d\P }\bigg|_{\F_s}=\frac{\rho_m(X_s)}{\rho_m(x)}\exp(-ms+L_s\phi(m)).
\end{equation}
The above is a combination of Doob's $h$-transform and a Feynman-Kac transform by local time $L_t$.
It follows that for $x>0$,
\begin{align*}
	\op^{(m)}f(x)=&\rho_m^{-1}(x)(\op-m)(\rho_m\cdot f)(x)\\
	=&\op f(x)+2a(x)\frac{\rho'_m(x)}{\rho_m(x)}f'(x)+\rho_m^{-1}(x)(\op-m)(\rho_m)(x)f(x)
\end{align*} 
Since $\rho_m(x)$ satisfies $(\op-m)\rho_m(x)=0$, 
under the new measure $\Q$, the diffusion process $X_t$ is a reflected diffusion on $[0, \infty)$ with generator 
$$\op^{(m)}=\op+2a(x)\frac{\rho'_m(x)}{\rho_m(x)}\frac{d}{d x}, \text{ for }x>0. $$
By \eqref{e:2.3},   for every $\lambda >0$, 
$$
\Ex^\Q e^{-\lambda S_t}=\frac{\Ex e^{-(\lambda +m)S_t}}{\exp(-t\phi(m))}=\exp\{-t(\phi(\lambda+m)-\phi(m))\}.
$$
This proves that  the Laplace exponent of $S^{(m)}_t$ is $\phi_m(\lambda)=\phi(\lambda+m)-\phi(m)$.
\qed

\begin{remark} \label{rmk:EG} \rm
By Feymann-Kac transformation, $\rho_m(x)=\Ex [\exp(-mT_0)]$ is the unique solution to 
$$\begin{cases}
(\op-m)\rho_m=0;\\
\rho_m(0)=1,\ \rho_m(\infty)=0.
\end{cases}$$
\end{remark}


\section{Comparison theorem for inverse local time}
\label{sec:com}

Let   $X_t$ and $Y_t$ be reflected diffusion processes on $[0, \infty)$ defined on a probability space $(\Omega, \F, \P)$ and driven by a 
common Brownian motion, whose  generators are 
\begin{align*}
&\op^X=a(x)\frac{d^2}{d x^2}+b(x)\frac{d}{d x};\\
&\op^Y=a(x)\frac{d^2}{d x^2}+B(x)\frac{d}{d x}.
\end{align*}
Denote by $Z^X$ and $ Z^Y$  the  zero sets for  $X$ and $Y$ respectively; that is,
$$
Z^X:=  {\{t\in[0,\infty): X_t=0\}} \quad \hbox{ and } \quad 
Z^Y:= {\{t\in[0,\infty): Y_t=0\}}.
$$
 These are random closed subsets of $[0, \infty)$, and are regenerate (also called Markov) sets
 in the sense of Maisonneuve (cf \cite{B}).

 \begin{lemma}
\label{lem:lap}
Let  $S_t^X$ and $ S_t^Y$ be inverse local times at $0$ for $X_t, Y_t$,   whose Laplace exponents are denoted by 
$\phi^X(\lambda)$, $\phi^Y(\lambda)$, respectively. 
Suppose that $b(x)\leq B(x)$ for all $x$. 
Then $\phi^Y/\phi^X$ is a completely monotone function.
\end{lemma}

\proof  If   $X_0\leq Y_0$, then by the comparison theorem for one-dimensional diffusions (see, e.g., 
\cite[Theorem I.6.2]{Bass}), 
we have, almost surely, $X_t\leq Y_t$ for all $t\geq 0$. 
Consequently,   \begin{equation}
\label{eqn:z}
Z^Y\subset Z^X, \ \P\text{-a.s.}
\end{equation}
Note that  $S_t^X$ and $ S_t^Y$ are the subordinators associated with the regenerative sets $Z^X$ and $Z^Y$, respectively.
It follows from the regenerative embedding theorem due to Bertoin (see  \cite[Theorem 1]{B}) that
  $\phi^Y/\phi^X$ is a completely monotone function. \qed

\begin{example} \label{ex:stablecom}   \rm
Consider two reflected Bessel processes on $[0, \infty)$: $X_t^{(\alpha)}, X_t^{(\beta)}$,   determined by local generators:
\begin{align*}
&\op^{(\alpha)}=\frac{1}{2}\frac{d^2}{d x^2}+\frac{1-2\alpha}{2x}\frac{d}{d x};\\
&\op^{(\beta)}=\frac{1}{2}\frac{d^2}{d x^2}+\frac{1-2\beta}{2x}\frac{d}{d x},
\end{align*}
where $0<\beta<\alpha<1$ and $X_0^{(\alpha)}\leq X_0^{(\beta)}$. Denote by $S_t^{(\alpha)}, S_t^{(\beta)}$ the inverse local times at $0$, and $\phi^{(\alpha)}, \phi^{(\beta)}$ their Laplace exponents. 

Since $\frac{1-2\alpha}{2x}<\frac{1-2\beta}{2x}$ for $x>0$, apply the classic comparison theorem for one-dimensional SDE
and Lemma \ref{lem:lap}, $\phi^{(\beta)}/\phi^{(\alpha)}$ is completely monotone.

On the other hand,  as we know from \cite{MO}, $S_t^{(\alpha)}, S_t^{(\beta)}$ are $\alpha$- and $\beta$-stable subordinators, respectively. 
Since $\phi^{(\alpha)}(\lambda)=c_\alpha \lambda^\alpha$, $\phi^{(\beta)}(\lambda)=c_\beta\lambda^\beta$,   $\phi^{(\beta)} (\lambda )/\phi^{(\alpha)} (\lambda ) 
=(c_\beta/c_\alpha) \lambda^{\beta -\alpha}$ is indeed completely monotone in $\lambda$.
\end{example}

In general one can not conclude that the inverse local time at 0 of $Y$ is dominated by that of $X$. 
Indeed, there is no monotonicity between $\alpha$-stable and $\beta$-stable subordinators,
as there is no monotonicity between their Laplace exponents. However we have the following
comparison theorem for inverse local times. 

 \begin{thm}
\label{thm:com}
Suppose  $X_t$ and $Y_t$ defined on a probability space $(\Omega,\F, \mathbb P)$ are reflected diffusions on $[0, \infty)$,
 determined by the local generator
\begin{align*}
&\op^X=\frac{1}{2}\frac{d^2}{d x^2}+b(x)\frac{d}{d x};\\
&\op^Y=\frac{1}{2}\frac{d^2}{d x^2}+B(x)\frac{d}{d x}.
\end{align*}
Let  $S_t^X$ and $S_t^Y$ be the corresponding inverse local times at $0$,  respectively.  
Suppose  $f(x)=B(x)-b(x) \geq 0$ satisfies the condition
\begin{equation}
\label{con:Gir}
\sup_{x>0}\Ex_x\bigg[\int_0^{T}|f(X_t)|^2d t\bigg]<\infty,\ \text{ for any fixed time  }T>0.
\end{equation}
 Then stochastically,
 $S_t^X\leq S_t^Y$ for all $t\geq 0$.
\end{thm}

\proof 
We first define a Girsanov transform between $X_t$ and $Y_t$,
$$
\frac{d \Q}{d \P}\bigg|_{\F_t}=\exp\bigg[\int_0^t f(X_s)d B_s-\frac{1}{2}\int_0^tf^2(X_s)d s\bigg].
$$
Note that due to condition \eqref{con:Gir}, by \cite[Theorem 3.2]{Chen},   the right hand side   of the above is a uniformly integrable martingale. 
 Thus with $Z^X$, $Z^Y$ as zero sets, we have the relations
\begin{equation}
\label{eqn:zero}
(X_t,\Q)\overset{d}{=}(Y_t,\P)\Rightarrow (Z^X,\Q)\overset{d}{=}(Z^Y,\P).
\end{equation}
In other words, under $\Q$, $X_t$ can be viewed as $Y_t$. This leads to the same properties for zero sets of $X_t$ and $Y_t$.

Now let  $L_t^X$ be a choice of the local time for $X_t$ at $0$ such that $L_t^X$ satisfies (cf. Theorem X.2. in \cite{M}) 
\begin{equation}
\label{eqn:lt}
\Ex_x\bigg[\int_t^\infty e^{-s}d L_s^X\bigg |\F_t\bigg]=\Ex_x\big[e^{-T_0\circ\theta_t}|\F_t\big],
\end{equation}
where $T_0$ is the first hitting time at $0$ for $X_t$. We claim that 
$M_t\triangleq\{e^{-T_0}\1_{\{T_0\leq t\}}-\int_0^te^{-s}d L_s^X;t\geq 0\}$ is a $\P$-martingale with respect to the filtration
$\{\F_t; t\geq 0\}$.   This is because for every $t\geq r$, 
\begin{align*}
\Ex_x[M_t|\F_r]=&M_r+\Ex_x\bigg[e^{-T}\1_{\{r<T\leq t\}}-\int_r^te^{-s}d L_s^X\bigg|\F_r\bigg]\\
=&M_r+\Ex_x\bigg[e^{-T}\1_{\{r<T\leq t\}}-\int_r^te^{-s}d L_s^X\bigg|\F_r\bigg]+\Ex_x\bigg[e^{-T\circ\theta_t}-\int_t^\infty e^{-s}d L_s^X\bigg |\F_r\bigg]\\
=&M_r+\Ex_x\bigg[e^{-T\circ\theta_r}-\int_r^\infty e^{-s}d L_s^X\bigg |\F_r\bigg]\\
=&M_r, 
\end{align*}
where the last equality is due to \eqref{eqn:lt}. 
This proves the claim that 
$\{M\}_t$ is a martingale with respect to $\{\F_t\}_{\{t\geq0\}}$. Clearly, it is purely discontinuous martingale of finite variation. 

Applying the same Girsanov transform to $M_t$, $M_t-[M,\int_0^\cdot f(X_s)d B_s]_t$ is a $\Q$-martingale. Since $M_t$ is a purely discontinuous martingale of finite variation and $\int_0^tf(X_s)d B_s$ is continuous, $[M,\int_0^\cdot f(X_s)d B_s]_t=0$, $M_t$ is a $\Q$-martingale as well. Thus,
$$\Ex_x^\Q [e^{-T_0}\1_{\{T_0\leq t\}}]=\Ex_x^\Q\bigg[\int_0^te^{-s}d L_s^X\bigg]$$
By letting $t\to\infty$, one can have
$$\Ex_x^\Q[e^{-T_0}]=\Ex_x^\Q\bigg[\int_0^\infty e^{-s}d L_s^X\bigg].$$
Thus, we  get the relation that $(L_t^X, \mathbb Q)\overset{d}{=}(L_t^Y,\mathbb P)$.

Fristedt-Pruitt showed in \cite{FP} there exists an increasing function $g$ such that 
$$g\text{-}m(S^X[0,t])=t,  $$
where the left hand side represents the Hausdorff measure of the range of $S^X$ on the time interval $[0,t]$ with respect to the function $g$. Since the closure of the range for $S^X$ is $Z^X$, it follows that $g\text{-}m(Z^X\cap[0,t])=L_t^X$, $\mathbb P\text{-a.s.}$ and 
$$(g\text{-}m(Z^X\cap[0,t]);\Q)=(L_t^X,\Q)\overset{d}{=}(L_t^Y,\P).$$
Also, by the classic comparison theorem, $X_t\leq Y_t$ almost surely for all $t$, we  have $Z^X\supset Z^Y$, $\P$-a.s. Together with (\ref{eqn:zero}), 
\begin{align*}
(L_t^Y,\P)=(g\text{-}m(Z^X\cap[0,t]);\Q)=&(g\text{-}m(Z^Y\cap[0,t]);\P)\\
\leq& (g\text{-}m(Z^X\cap[0,t]);\P)\\
=&(L_t^X,\P)
\end{align*}
The conclusion of the theorem now follows. 
\qed

Denote by $\mu_X$ and $\mu_Y$ the L\'evy measure for the subordinators $S^X_t$ and $S^Y_t$, respectively. 
The following is a comparison theorem on L\'evy measures.

\begin{thm}
\label{T:3.4}
Suppose $X_t$ and $Y_t$ are reflected diffusions on $[0, \infty)$ as  in Theorem \ref{thm:com}, $\phi^{X}$ and $\phi^{Y}$ are the Laplace exponents of inverse local times, respectively. Then $\phi^Y-\phi^X \geq 0$ is completely monotone and, consequently, $\mu_X \leq \mu_Y$.
\end{thm}

\begin{proof}
Applying Theorem \ref{thm:com}, we have $S_t^X\leq S_t^Y$, $\P$-a.s.  and so $0\leq\phi^X\leq\phi^Y$. \\
On the other hand, since $b(x)\leq B(x)$,  by  Lemma \ref{lem:lap}, $\phi^Y/\phi^X$ is completely monotone. Combining the two facts,
we see that 
$$
\frac{\phi^Y}{\phi^X}-1\geq 0\text{ is completely monotone.}
$$
Since completely monotone relation is preserved under multiplication (check details in Chapter 1, \cite{SSV}), we would have 
$$
\phi^Y-\phi^X= \left(\frac{\phi^Y}{\phi^X}-1\right)\phi^X\geq0\text{ is completely monotone.}
$$
This says that $\phi^Y-\phi^X$ is the Laplace exponent of some  subordinator $Z$. 
Hence $S^Y$ has the same distribution as the independent sum of two subordinators $S^X$ and $Z$.
Denote by $\nu$ the L\'evy measure for $Z$. 
It follows then $\mu_Y  - \mu_X =\nu\geq 0$. 
\qed
\end{proof}
 
\section{Properties of non-local operators}
\label{S:4}

We use the same notations as in Example \ref{ex:stablecom}.
For $0<\alpha<1$, let  $X^{(\alpha)}_t$ is a reflected Bessel process on $[0,\infty)$
with the local generator
$$
\op^{(\alpha)}=\frac{1}{2}\frac{d^2}{d x^2}+\frac{1-2\alpha}{2x}\frac{d}{d x}.
$$
As we noted earlier, the inverse local time $S^{(\alpha)}_t$ is an $\alpha$-stable subordinator, 
with the Laplace exponent
\begin{equation}
\label{eqn:stable}
\phi^{(\alpha)}(\lambda)=c_\alpha\lambda^{\alpha}.
\end{equation}
We know from Theorem \ref{thm:EG} that under the new probability measure $\Q$ defined by \eqref{e:2.3},
 the inverse local time of the Girsanov transformed diffusion, $X_t^{(\alpha,m)}$, is a relativistic $\alpha$-stable subordinator, with the Laplace exponent
\begin{equation}
\label{eqn:rela}
\phi^{(\alpha,m)}(\lambda)=\phi^{(\alpha)}(\lambda+m)-\phi^{(\alpha)}(m)=c_\alpha\big((\lambda+m)^{\alpha}-m^\alpha\big).
\end{equation}
The new reflected diffusion $X_t^{(\alpha,m)}$ on $[0, \infty)$ has generator 
\begin{equation}
\label{eqn:relabessel}
\op^{(\alpha,m)}=\frac{1}{2}\frac{d^2}{dx^2}+\bigg(\frac{1-2\alpha}{2x}+\frac{\rho'_m(x)}{\rho_m(x)}\bigg)\frac{d}{dx}, \text{ for }x>0,
\end{equation}
where $\rho_m (x):= \Ex_x \left[ \exp (-m T_0 )\right]$ with $T_0$ being the first hitting time of $0$ by $X^{(\alpha)}$. 
By Remark \ref{rmk:EG}, $\rho_m(x)$ is the unique solution to 
$$\begin{cases}
(\op^{(\alpha)}-m)\rho(x)=0,\\
\rho(0)=1;\ \rho(\infty)=0.
\end{cases}$$
It is know from  ODE, 
$$
\rho_m(x)=\wh c_\alpha \wh K_\alpha(\sqrt{2m}x) , 
$$
where $\wh c_\alpha$ is a normalizing constant depending on $\alpha$ only, $\wh K_\alpha=x^\alpha K_\alpha$ and
\begin{equation}
	\label{eqn:Kalpha}
	K_\alpha(x)=\frac{\pi}{2}\frac{I_{-\alpha}(x)-I_\alpha(x)}{\sin(\alpha\pi)},
\end{equation}
where 
$$I_\alpha(x)=\sum_{n=0}^\infty\frac{1}{n!\Gamma(n+\alpha+1)}\left(\frac{x}{2}\right)^{2n+\alpha}.$$
The function $I_\alpha$ is a   solution to  the following 
 modified Bessel's equation
$$x^2u''(x)+xu'(x)-(x^2+\alpha^2)u=0.
$$ 
Clearly,  $K_\alpha$ also satisfies the above equation,  and is called a modified Bessel function of the second kind.

\begin{example} \label{ex:Cauchy} \rm 
Suppose $\alpha=0.5$. Note that  
\begin{align*}
K_{0.5}(x)=&\frac{\pi}{2}\sum_{n=0}^{\infty}\frac{(\frac{x}{2})^{2n}}{n!}\bigg[\frac{(\frac{x}{2})^{-0.5}}{\Gamma(0.5+n)}-\frac{(\frac{x}{2})^{0.5}}{\Gamma(1.5+n)}\bigg]\\
=&\frac{\pi}{\sqrt{2x}}\sum_{n=0}^\infty\bigg[\frac{(\frac{x}{2})^{2n}}{n!\Gamma(0.5+n)}-\frac{(\frac{x}{2})^{2n+1}}{n!\Gamma(1.5+n)}\bigg]\\
=&\frac{\pi}{\sqrt{2x}}\sum_{n=0}^\infty\bigg[\frac{x^{2n}}{\Gamma(0.5)(2n)!}-\frac{x^{2n+1}}{\Gamma(0.5)(2n+1)!}\bigg]\\
=&\frac{\pi e^{-x}}{\Gamma(0.5)\sqrt{2x}}.
\end{align*}
Thus for $\alpha=0.5$,
$$
\rho_m(x)=\frac{\pi}{\Gamma(0.5)\sqrt{2}}\exp(-\sqrt{2m}x).
$$
Consequently, we have the perturbation part as
$$\frac{\rho'_m(x)}{\rho_m(x)}=-\sqrt{2m}.$$
Hence if $X^{(0.5,m)}_t$ is a reflected process on $[0,\infty)$  with the local generator
$$\op^{(0.5,m)}=\frac{1}{2}\frac{d^2}{d x^2}-\sqrt{2m}\frac{d}{d x},$$
then its inverse local time at $0$ is a relativistic Cauchy subordinator  with the Laplace exponent
$$
\phi^{(0.5,m)}(x)=c(\sqrt{\lambda+m}-\sqrt{m}).
$$
\end{example}

\begin{example}
Now if we operate another Girsanov transform on $\op^{(\alpha,m)}$, that is
$$\op^{(1)}=\op^{(\alpha,m)}+\frac{q'_n(x)}{q_n(x)}\frac{d}{d x},$$
where $q_n(x)$ is the unique solution to 
$$\begin{cases}
(\op^{(\alpha,m)}-n)q_n(x)=0;\\
q_n(0)=1,\ q_n(\infty)=0.
\end{cases}$$
Then the inverse local time at $0$, $S_t^{(1)}$, of the new reflecting diffusion generated by the above generator, has the Laplace exponent
$$\phi^{(1)}(\lambda)=\phi^{(\alpha,m)}(\lambda+n)-\phi^{(\alpha,m)}(n)=c_\alpha\big((\lambda+m+n)^\alpha-(m+n)^\alpha\big).$$
It can also be viewed as the Laplace exponent of a relativistic $\alpha$-stable subordinator with mass $m+n$, which is obtained as the inverse local time at $0$ for a reflecting diffusion determined locally by 
$$\op^{(\alpha,m+n)}=\op+\frac{\rho'_{m+n}(x)}{\rho_{m+n}(x)}\frac{d}{d x}.$$
The two generators should be the same, so we get the relation
$$\frac{\rho'_m(x)}{\rho_m(x)}+\frac{q'_n(x)}{q_n(x)}=\frac{\rho'_{m+n}(x)}{\rho_{m+n}(x)}$$
\end{example}

\begin{remark} S. Watanabe \cite{W} has defined a conservative diffusion process $\wt X_t$ on $[0,\infty)$, determined by the local generator in the same form
	$$\wt\op=\frac{1}{2}\frac{d^2}{dx^2}+\left(\frac{1-2\alpha}{2x}+\frac{\rho_c'(x)}{\rho_c(x)}\right)\frac{d}{dx},$$
but with $\rho_c(x)$ as the unique solution to
$$\begin{cases}
(\op^{(\alpha)}-c)\rho(x)=0;\\
\rho(0)=1,\ \rho'(0)=0.
\end{cases}$$	
Thus, the generator can be written as
$$\wt\op u=\frac{1}{\rho_c}(\op^{(\alpha)}-c)(\rho_c\cdot u)(x),$$
and this yiels an explicit expression of the transition density for $X_t^{(\alpha,c)}$
$$\wt p(t,x,y)=\frac{e^{-mt}p^{(\alpha)}(t,x,y)}{\rho_c(x)\rho_c(y)},\ x,y\geq 0,$$
where $p^{(\alpha)}(t,x,y)$ is the transition density of the Bessel process $X_t^{(\alpha)}$ with respect to the measure $m^{(\alpha)}(dx)=x^{1-2\alpha}dx$.
\end{remark}
We continue to discuss the reflecting diffusions, $X_t^{(\alpha,m)}$, which is locally determined by the generator (\ref{eqn:relabessel}). Consider the drift term $\frac{\rho'_m(x)}{\rho_m(x)}$, because $K'_\alpha(x)=-\frac{\alpha}{x}K_\alpha(x)-K_{\alpha-1}(x)$
\begin{align}
\label{eqn:rho}
\frac{\rho'_m(x)}{\rho_m(x)}=&\sqrt{2m}\frac{\wh K'_\alpha(\sqrt{2m}x)}{\wh K_\alpha(\sqrt{2m}x)}\notag\\
=&\frac{\alpha}{x}+\sqrt{2m}\frac{K'_\alpha(\sqrt{2m}x)}{K_\alpha(\sqrt{2m}x)}\notag\\
=&\frac{\alpha}{x}+\sqrt{2m}\frac{-\frac{\alpha}{\sqrt{2m}x}K_\alpha(\sqrt{2m}x)-K_{\alpha-1}(\sqrt{2m}x)}{K_\alpha(\sqrt{2m}x)}\notag\\
=&-\sqrt{2m}\frac{K_{\alpha-1}(\sqrt{2m}x)}{K_{\alpha}(\sqrt{2m}x)}.
\end{align}
We will focus on the asymptotic behaviors of $\frac{\rho'_m(x)}{\rho_m(x)}$ near $0$ and $\infty$ and have the following lemma:
\begin{lemma} For $m\geq 0$, $0<\alpha<1$,
\begin{equation}
\label{eqn:asymp}
\frac{\rho'_m(x)}{\rho_m(x)}=-\sqrt{2m}\frac{K_{\alpha-1}(\sqrt{2m}x)}{K_{\alpha}(\sqrt{2m}x)}\sim\begin{cases}
-\frac{m^\alpha\Gamma(1-\alpha)}{2^{\alpha-1}\Gamma(\alpha)}x^{2\alpha-1}\text{ as }x\to 0+;\\
 -\sqrt{2m}\text{ as }x\to\infty,
\end{cases}
\end{equation}
where $\sim$ means the ratio between two sides approaches $1$ as $x$ goes to $0+$ or $\infty$.
\end{lemma}
\proof When $x\to 0+$ and $\nu\notin\Z$, $K_\nu(x)$ has the following series expansion:
\begin{align*}
K_\nu(x)\propto \frac{1}{2}\bigg(\Gamma&(\nu)\Big(\frac{x}{2}\Big)^{-\nu}\Big(1+\frac{x^2}{4(1-\nu)}+\frac{x^4}{32(1-\nu)(2-\nu)}+\cdots\Big)\\
&+\Gamma(-\nu)\Big(\frac{x}{2}\Big)^\nu\Big(1+\frac{x^2}{4(\nu+1)}+\frac{x^4}{32(\nu+1)(\nu+2)}+\cdots\Big)\bigg).
\end{align*}
Since $0<\alpha<1$, $\alpha,\alpha-1\notin\Z$ in (\ref{eqn:rho}), as $x\to 0+$, 
\begin{align}
\label{eqn:asym0}
\frac{\rho'_m(x)}{\rho_m(x)}\sim&-\sqrt{2m} \frac{\Gamma(1-\alpha)\Big(\frac{\sqrt{2m}x}{2}\Big)^{\alpha-1}}{\Gamma(\alpha)\Big(\frac{\sqrt{2m}x}{2}\Big)^{-\alpha}}\notag\\
\sim&-\frac{m^\alpha\Gamma(1-\alpha)}{2^{\alpha-1}\Gamma(\alpha)}x^{2\alpha-1}.
\end{align}
When $x\to\infty$, $K_\nu(x)$ can be described as the following formula:
$$K_\nu(x)\propto \sqrt{\frac{\pi}{2}}\frac{e^{-x}}{\sqrt{x}}\bigg(1+O\Big(\frac{1}{x}\Big)\bigg).$$
The asymptotic behavior near $\infty$ is independent of the index $\nu$. Thus, as $x\to\infty$
\begin{equation}
\label{eqn:asyminf}
\frac{\rho'_m(x)}{\rho_m(x)}\sim -\sqrt{2m}.
\end{equation}
\qed 

We are now in the position to present the proof for   Theorem \ref{T:1.1}.

\noindent\textbf{Proof of Theorem \ref{T:1.1}.} $Y_t$ is a reflecting diffusion process, determined by the local generator
$$\op=\frac{1}{2}\frac{d^2}{d x^2}+\left(\frac{1-2\alpha}{2x}-f(x)\right)\frac{d}{d x},$$
and there exists a constant $c_1$ such that 
$$0\leq f(x)\leq c_1(1\wedge x)^{2\alpha-1}.
$$
Now we check the condition (\ref{con:Gir}) in Theorem \ref{thm:com}, $f(x)$ is bounded when $1/2\leq\alpha<1$, so the condition is naturally satisfied. When $0<\alpha< 1/2$, for a fixed $T>0$, with $p^{(\alpha)}(t,x,y)$ representing the transition density of a Bessel process of index $\alpha$ with respect to the measure $m^{(\alpha)}(dx)=2x^{1-2\alpha}dx$, 
\begin{align*}
\sup_{x>0}\Ex_x\bigg[\int_0^T|f(X^{(\alpha)}_t)|^2\rd t\bigg]\leq&c_1\sup_{x>0}\int_{(0,\infty)}\bigg(\int_0^T p^{(\alpha)}(t,x,y)(1\vee y^{4\alpha-2})d t\bigg)m^{(\alpha)}(dy)\\
\leq&c_1T+c_1\sup_{x>0}\int_0^1\bigg(\int_0^T \frac{x^\alpha y^{3\alpha-1}}{t}\exp\left(-\frac{x^2+y^2}{2t}\right)I_{-\alpha}\Big(\frac{xy}{t}\Big)d t\bigg)dy\\
=&c_1T+c_1\sup_{x>0}\int_0^1\int_{xy/T}^\infty \frac{x^\alpha y^{3\alpha-1}}{s}\exp\left(-\frac{(x^2+y^2)s}{2xy}\right)I_{-\alpha}(s)dsdy\\
\leq&c_1T+c_1\sup_{x>0}\int_0^1x^\alpha y^{3\alpha-1}\bigg(\int_{xy/T}^\infty \frac{e^{-s}}{s}I_{-\alpha}(s)ds\bigg)dy,
\end{align*} 
where $I_{-\alpha}(s)$ is the modified Bessel function of the first kind. Then
$$I_{-\alpha}(s)\propto\begin{cases}
\frac{1}{\Gamma(1-\alpha)}\Big(\frac{s}{2}\Big)^{-\alpha}\bigg(1+\frac{s^2}{4(1-\alpha)}+\frac{s^4}{32(1-\alpha)(2-\alpha)}+\cdots\bigg),\ s\to0;\\
\frac{e^s}{\sqrt{2\pi s}}\bigg(1+O\Big(\frac{1}{s}\Big)\bigg),\ s\to\infty.
\end{cases}$$
Since we have $0<y<1$, by the above asymptotic behavior
$$\int_{xy/T}^\infty \frac{e^{-s}}{s}I_{-\alpha}(s)ds\text{ is dominated by } C_\alpha T^\alpha (xy)^{-\alpha} \text{ as }x\to 0+; \text{ by }C_\alpha T^{1/2}(xy)^{-1/2}\text{ as }x\to\infty.$$
We would then get for $0<\alpha<1/2$
$$\sup_{x>0}\Ex_x\bigg[\int_0^T|f(X^{(\alpha)}_t)|^2\rd t\bigg]<\infty.$$
Thus, one can set up a Girsanov transform between $X_t$ and $X_t^{(\alpha)}$, or, $X_t$ and $X_t^{(\alpha)}$ are absolutely continuous to each other. Applying Theorem \ref{thm:com}, we have $S_t\leq S_t^{(\alpha)}$, $\P$-a.s.

For any $0<\alpha<1$, from the asymptotic behaviors of $\frac{\rho'_m(x)}{\rho_m(x)}$ near $0$ and $\infty$ shown in (\ref{eqn:asymp}), we can always choose a proper value of $m$ such that
$$c_1\leq\sqrt{2m}\wedge\frac{m^\alpha\Gamma(1-\alpha)}{2^{\alpha-1}\Gamma(\alpha)},$$
consequently,
$$0\leq f(x)\leq -\frac{\rho'_m(x)}{\rho_m(x)},\ m\text{ depends on }c_1.$$
Thus, by the classic Comparison theorem, 
\begin{equation}
X_t^{(\alpha,m)}\leq Y_t\leq X_t^{(\alpha)},\ \P\text{-a.s.}
\end{equation}
By Theorem \ref{thm:EG}, $X_t^{(\alpha)}$ and $X_t^{(\alpha,m)}$ are absolutely continuous to each other. Thus, $Y_t$ and $X_t^{(\alpha,m)}$ are absolutely continuous to each other, i.e., there exists a Girsanov transform between them. Applying Theorem \ref{thm:com} again, $S_t^{(\alpha,m)}\leq S_t$, $\P$-a.s.
\qed

To prove Theorem \ref{T:1.2}, we first recall the following result from   Grzywny-Ryznar \cite{GR} (with slightly different notation here).

\begin{thm} \label{thm:GR} {\rm  \cite[Theorem 1.1]{GR} }
Let  $D\subset\R^d$ be a bounded Lipschitz open set. Suppose that $Y_t$ is a symmetric L\'evy  process on $\R^d$ with L\'evy measure  $\nu(x) dx$.
Denote  by $\nu^{(\alpha)}(x)$ the L\'evy density for  the isotropic $\alpha$-stable process $Z$ on $\R^d$. 
Denote by $G_D$ and $G^{(\alpha)}_D$ the Green functions of $Y$ and $Z$  in $D$, respectively. 
Assume that $j(x)=\nu^{(\alpha)}(x)-\nu(x)\geq0$ on $\R^d$, and that $j(x)\leq c|x|^{\rho-d}$ for $|x|\leq 1$, where $c,\rho>0$. Then there exists a constant $C=C(d,\alpha, D, \rho, c)$, such that
$$
C^{-1} G_D^{(\alpha)}(x,y)\leq G_D(x,y)\leq CG_D^{(\alpha)}(x,y) \quad 
\hbox{for all } x,y\in D.
$$
\end{thm}

\bigskip

\noindent\textbf{Proof of Theorem \ref{T:1.2}.} 
Denote the Laplace exponents of $S_t$, $S_t^{(\alpha)}$, $S_t^{(\alpha,m)}$, by $\phi(\lambda)$, $\phi^{(\alpha)}(\lambda)$, $\phi^{(\alpha,m)}(\lambda)$, respectively.
Now applying Theorem \ref{T:3.4}, 
$$\phi(\lambda)-\phi^{(\alpha,m)}(\lambda)\text{ and } \phi^{(\alpha)}(\lambda)-\phi(\lambda)\text{ are completely monotone.}$$
Denote $\nu^{(\alpha,m)}, \nu,\nu^{(\alpha)}$ as \Levy measures of inverse local times respectively, then 
$$\nu^{(\alpha)}-\nu\geq 0;\ \nu-\nu^{(\alpha,m)}\geq 0$$
and for any $0<\alpha<1$, $t>0$,
\begin{equation}
\label{eqn:mea}
0\leq(\nu^{(\alpha)}-\nu)(t)\leq (\nu^{(\alpha)}-\nu^{(\alpha,m)})(t)\leq c_\alpha\frac{1-e^{-mt}}{t^{\alpha+1}}.
\end{equation}
Thus, for $|x|\leq1$, the difference between \Levy measures of trace processes,  $B_{S_t}, B_{S_t^{(\alpha)}}$, would be
\begin{align*}
j(x)=&c_\alpha\int_0^\infty (4\pi t)^{-d/2}e^{-\frac{|x|^2}{4t}}(\nu^{(\alpha)}-\nu)(t)dt\\
=&c_\alpha\int_0^\infty (4\pi t)^{-d/2}e^{-\frac{|x|^2}{4t}}\frac{1-e^{-mt}}{t^{\alpha+1}}dt\\
\leq&c_\alpha\int_0^\infty (4\pi t)^{-d/2}e^{-\frac{|x|^2}{4t}}\frac{mt}{t^{\alpha+1}}dt\\
=&c_\alpha\pi m^{-d/2}4^{\alpha-1}|x|^{-d+2-2\alpha}\int_0^\infty s^{d/2+\alpha-2}e^{-s}ds\\
\leq &C|x|^{-d+2-2\alpha},
\end{align*}
where $C=C(\alpha,m,d,c_\alpha)$. The second to the last equality is obtained by doing change of variables $s=|x|^2/(4t)$. From the proof of Theorem \ref{T:1.1}, $m$ is fixed once given a $f(x)$. 

Applying Theorem \ref{thm:GR} with $\rho=2-2\alpha>0$, we conclude that there exists a constant $C_1=C(d,\alpha, D,C)$ such that
$$C_1^{-1}G_D^{(2\alpha)}(x,y)\leq G_D(x,y)\leq C_1G_D^{(2\alpha)}(x,y)
$$
for all $x,y\in D$.\qed

\bigskip

{\bf Acknowledgement.} We thank P. J. Fitzsimmons, M. M.  Meerschaert and Z. Vondracek for helpful discussions.

\vskip 0.3truein 

Department of Mathematics, University of Washington, Seattle,
WA 98195, USA

Email: \texttt{zqchen@uw.edu}

Email: \texttt{lidanw@uw.edu}

\end{document}